\documentclass[11pt]{amsart}

\usepackage{algorithm}
\usepackage{algpseudocode}
\usepackage{epigraph}
\usepackage{mathtools}
\usepackage{latexsym}
\usepackage{mathrsfs}
\usepackage{pstricks}
\usepackage{listings}
\usepackage{hyperref}
\usepackage{pgfplots}
\usepackage{appendix}
\usepackage{pst-all}
\usepackage{pstricks,pst-plot}
\pgfplotsset{compat=1.18} 
\usepackage{mathdots}
\usepackage{xcolor}
\usepackage{enumitem}
\usepackage{todonotes}
\usepackage{amssymb,bm}
\usepackage{amsmath}
\usepackage{amsfonts}
\usepackage{amsthm}
\usepackage{tikz}
\usepackage{epsfig}
\usepackage{verbatim}
\usepackage{float}
\usepackage{tabularx}
\usepackage{bm}
\usepackage{mathtools}
\usepackage{blkarray, bigstrut}
\usepackage{subcaption}

\definecolor{codegray}{rgb}{0.5,0.5,0.5}
\lstdefinestyle{mystyle}{
    commentstyle=\color{codegreen},
    keywordstyle=\color{magenta},
    numberstyle=\tiny\color{codegray},
    stringstyle=\color{codepurple},
    basicstyle=\ttfamily\footnotesize,
    breakatwhitespace=false,         
    breaklines=true,                 
    captionpos=b,                    
    keepspaces=true,                 
    numbers=left,                    
    numbersep=5pt,                  
    showspaces=false,                
    showstringspaces=false,
    showtabs=false,                  
    tabsize=2
}
  \lstdefinelanguage{GAP}{
    basicstyle=\ttfamily,
    keywords={true, false, function, return, fail, if, in, while, do, od, else, elif, fi, break, continue},
    keywordstyle=\color{blue}\bfseries,
    otherkeywords={
      >, <, ==
    },
    breaklines=true,      
    identifierstyle=\color{black},
    sensitive=True,
    comment=[l]{\#},
    commentstyle=\color{cyan},
    stringstyle=\color{red},
    morestring=[b]',
    morestring=[b]"
  }

\providecommand{\U}[1]{\protect\rule{.1in}{.1in}}

\newcolumntype{Y}{>{\raggedleft\arraybackslash}X}
\def\bc{{\mathbb{C}}}

\def\bn{{\mathbb{N}}}

\def\br{{\mathbb{R}}}

\def\bz{{\mathbb{Z}}}
\def\bt{{\mathbb{T}}}

\def\br{\mathbb R}

\def\wt{\widetilde}

\def\vs{\vskip.3cm}
\def\noi{\noindent}
\def\wt{\widetilde}

\def\t2deg{\mathbb T^2\text{\rm -deg}}
\def\tndeg{\mathbb T^n\text{\rm -deg}}

\def\s1deg{S^1\text{\rm -deg}}

\def\Om{\Omega}

\def\sign{\text{\rm sign\,}}

\def\ker{\text{\rm Ker\,}}

\DeclareMathOperator{\id}{Id}

\newrgbcolor{violet}{.6 .1 .8}
  \newrgbcolor{lightyellow}{1 1 .8}
  \newrgbcolor{lightblue}{.80 1 1}
  \newrgbcolor{mygreen}{0 .66 .05}
  \definecolor{mygreen}{rgb}{0,.66,.05}
  \definecolor{lightyellow}{rgb}{1,1,.80}
  \newrgbcolor{orange}{1 .6 0}
  \newrgbcolor{GreenYellow}{.85 1 .31}
  \newrgbcolor{Yellow}{1  1  0}
  \newrgbcolor{Goldenrod}{1  .90  .16}
  \newrgbcolor{Dandelion}{1  .71  .16}
  \newrgbcolor{Apricot}{1  .68  .48}
  \newrgbcolor{Peach}{1  .50  .30}
  \newrgbcolor{Melon}{1  .54  .50}
  \newrgbcolor{YellowOrange}{1  .58  0}
  \newrgbcolor{Orange}{1  .39  .13}
  \newrgbcolor{BurntOrange}{1  .49  0}
  \newrgbcolor{Bittersweet}{1.  .4300  .24}
  \newrgbcolor{RedOrange}{1  .23  .13}
  \newrgbcolor{Mahogany}{1.  .4475  .4345}
  \newrgbcolor{Maroon}{1.  .4084  .5376}
  \newrgbcolor{BrickRed}{1.  .3592  .3232}
  \newrgbcolor{Red}{1  0  0}
  \newrgbcolor{OrangeRed}{1  0  .50}
  \newrgbcolor{RubineRed}{1  0  .87}
  \newrgbcolor{WildStrawberry}{1  .04  .61}
  \newrgbcolor{CarnationPink}{1  .37  1}
  \newrgbcolor{Salmon}{1  .47  .62}
  \newrgbcolor{Magenta}{1  0  1}
  \newrgbcolor{VioletRed}{1  .19  1}
  \newrgbcolor{Rhodamine}{1  .18  1}
  \newrgbcolor{Mulberry}{.6668  .1180  1.}
  \newrgbcolor{RedViolet}{.9538  .4060  1.}
  \newrgbcolor{Fuchsia}{.5676  .1628  1.}
  \newrgbcolor{Lavender}{1  .52  1}
  \newrgbcolor{Thistle}{.88  .41  1}
  \newrgbcolor{Orchid}{.68  .36  1}
  \newrgbcolor{DarkOrchid}{.60  .20  .80}
  \newrgbcolor{Purple}{.55  .14  1}
  \newrgbcolor{Plum}{.50  0  1}
  \newrgbcolor{Violet}{.98 .15 .95}
  \newrgbcolor{RoyalPurple}{.25  .10  1}
  \newrgbcolor{BlueViolet}{.84  .38  .98}
  \newrgbcolor{Periwinkle}{.43  .45  1}
  \newrgbcolor{CadetBlue}{.38  .43  .77}
  \newrgbcolor{CornflowerBlue}{.35  .87  1}
  \newrgbcolor{MidnightBlue}{.4414  .9259  1.}
  \newrgbcolor{NavyBlue}{.06  .46  1}
  \newrgbcolor{RoyalBlue}{0  .50  1}
  \newrgbcolor{Blue}{0  0  1}
  \newrgbcolor{Cerulean}{.06  .89  1}
  \newrgbcolor{Cyan}{0  1  1}
  \newrgbcolor{ProcessBlue}{.04  1  1}
  \newrgbcolor{SkyBlue}{.38  1  .88}
  \newrgbcolor{Turquoise}{.15  1  .80}
  \newrgbcolor{TealBlue}{.1572  1.  .6668}
  \newrgbcolor{Aquamarine}{.18  1  .70}
  \newrgbcolor{BlueGreen}{.15  1  .67}
  \newrgbcolor{Emerald}{0  1  .50}
  \newrgbcolor{JungleGreen}{.01  1  .48}
  \newrgbcolor{SeaGreen}{.31  1  .50}
  \newrgbcolor{Green}{0  1  0}
  \newrgbcolor{ForestGreen}{.1992  1.  .2256}
  \newrgbcolor{PineGreen}{.3100  1.  .5575}
  \newrgbcolor{LimeGreen}{.50  1  0}
  \newrgbcolor{YellowGreen}{.56  1  .26}
  \newrgbcolor{SpringGreen}{.74  1  .24}
  \newrgbcolor{OliveGreen}{.6160  1.  .4300}
  \newrgbcolor{RawSienna}{.53  .28  .16}
  \newrgbcolor{Sepia}{1.  .7510  .70}
  \newrgbcolor{Brown}{.41  .25  .18}
  \newrgbcolor{TAN}{.86  .58  .44}
  \newrgbcolor{Gray}{1.  1.  1.}
  \newrgbcolor{Black}{1  1  1}
  \newrgbcolor{White}{1  1  1}

\newtheorem{theorem}{Theorem}[section]

\newtheorem{lemma}{Lemma}[section]

\newtheorem{definition}{Definition}[section]

\newtheorem{remark}{Remark}[section]

\newtheorem{remark-definition}{Remark and Definition}[section]
\newtheorem{rem-not}{Remark and Notation}[section]

\begin{document}

\title[Global Bifurcation of Spiral Wave Solutions]{Global Bifurcation of Spiral Wave Solutions to the Complex Ginzburg-Landau Equation}

\author{Carlos Garcia-Azpeitia}\address{Departamento de Matem\'aticas y Mec\'anica, IIMAS-UNAM, Apdo. Postal 20-126, Col. San \'Angel,
Mexico City, 01000,  Mexico}
\email{cgazpe@mym.iimas.unam.mx}

\author{ Ziad Ghanem}\address{Department of Mathematical Sciences, University of Texas at Dallas, Richardson, TX 75080, USA}
\email{ziad.ghanem@UTDallas.edu}

\author{Wieslaw Krawcewicz}\address{Department of Mathematical Sciences, University of Texas at Dallas, Richardson, TX 75080, USA}
\email{wieslaw.@UTDallas.edu}

\date{}

\maketitle

\begin{abstract}
We use the $\mathbb T^2$-equivariant degree to establish the existence of unbounded branches of rotating spiral wave solutions with any prescribed number of arms for the complex Ginzburg Landau equation (GLe) on the planar unit disc. By leveraging spatial symmetries inherent to the problem, our approach avoids the restrictive assumptions required in previous studies \cite{Dai} that utilized the classical Leray-Schauder degree. Our results provide rigorous mathematical justification for the formation and persistence of these fundamental patterns, which are ubiquitous in physical, chemical, and biological systems but have, until now, eluded formal proof under general conditions.
\end{abstract}

\noi \textbf{Mathematics Subject Classification:} Primary: 37G40, 35B06; Secondary: 47H11, 35J57,  35J91

\medskip

\noi \textbf{Key Words and Phrases:} reaction-diffusion equation, symmetric bifurcation, unbounded branches,  spiral-wave solutions, equivariant degree.

\section{Introduction} \label{sec:introduction}
The tendency for spatially extended reaction-diffusion systems to exhibit spiral wave patterns has a long history of experimental verification with landmark examples including Belousov's 1951 observation of rotating and spiral formations in chemical cocktails of bromate, cerium, and acid (cf. \cite{Belousov}), Gerisch et al.'s 1974 discovery of spiral and circular growth patterns in slime mold cultures due to chemotaxis (cf. \cite{Gerisch}), and Allessie et al.'s 1977 demonstration of induced spiraling in the electrified atrial tissue of rabbits (cf. \cite{Allessie}). 
Originally introduced in the context of condensed-matter physics by the Russian physicists Vitaly Ginzburg and Lev Landau in the 1950s, the Ginzburg-Landau equation (GLe) serves as a unifying
phenomenological model of the nonlinear dynamics and pattern formation exhibited by many natural systems near critical points, such as phase transitions, superconductivity and superfluidity. For example, the evolution of a complex-valued wave function $\psi(t,r,\theta)$ associated with a superconducting condensate confined to the planar unit disc $D:= \{ z \in \bc : |z | < 1 \}$ and subjected to near absolute zero temperature can be modeled by the one-parameter $\alpha \in \br$ family of complex GLes:
\begin{align} \label{eq:system_GL}
\begin{cases}
\partial_t \psi = -(1+i \eta) \Delta \psi +  \alpha \psi + f(\psi), \quad \psi(t,r,\theta) \in \bc,  \\
\frac{\partial \psi}{\partial n}|_{\partial D} \equiv 0,
\end{cases}
\end{align}
where the bifurcation parameter $\alpha \in \br$ represents the growth (when positive) or decay (when negative) of the amplitude $|\psi|$, often associated with the proximity of the system to a phase transition, $\eta \in \br$ is a fixed diffusion parameter and $f:\bc \rightarrow \bc$ is a continuous function (typically, $f$ is chosen to be the cubic map $|u|^2u$) satisfying the conditions:
\begin{enumerate}[label=($A_\arabic*$)]
\item\label{a1} $f(e^{i \varphi} \psi) = e^{i \varphi} f(\psi)$ for all $\psi \in \bc$ and $\varphi \in S^1$; 
\item\label{a2} $f(\psi)$ is $o(|\psi|)$ as $\psi$ approaches $0$, i.e.
\[
    \lim_{\psi \rightarrow 0} \frac{f(\psi)}{|\psi|} = 0;
\]
\item\label{a3} there exist numbers $a,b >0$ and $c \in (0,1)$ such that
\[
|f(\psi)| < a|\psi|^c + b, \quad \psi \in \bc.
\]
\end{enumerate}
Relative equilibria of the GLe with a fixed temporal frequency $\beta \in \br$ and non-zero rotational frequency $\omega \in \br \setminus \{0\}$ are  solutions to \eqref{eq:system_GL} of the form $\psi(t,r,\theta) = e^{-i\beta t}u(r,\theta+ \omega t)$, which transform the boundary value problem \eqref{eq:system_GL} into the two-parameter $(\alpha,\beta) \in \br \times \br$ bifurcation problem 
\begin{align} \label{eq:system}
\begin{cases}
(1+i\eta)\Delta u + \omega \partial_\theta u = (\alpha + i \beta)u + f(u), \quad u(r,\theta) \in \bc,  \\
\frac{\partial u}{\partial n}|_{\partial D} = 0.  
\end{cases}
\end{align}
Condition \ref{a2} 
ensures the well-behavedness of our equation \eqref{eq:system}
at the origin in $\bc$, condition \ref{a3} is necessary to guarantee {\it a priori} bounds on its solutions and \ref{a1} implies that \eqref{eq:system} admits the symmetries of the Torus group
\[
\bt^2 := S^1 \times S^1.
\]
Special attention must be paid to the symmetry groups of the GLes \eqref{eq:system_GL} and \eqref{eq:system} --- arising from the domain of interest and the constraints placed on the nonlinear interaction term --- which regulate the symmetric properties of their possible solutions. Indeed, using a center bundle reduction framework, Golubitsky et al. examine in \cite{Golubitsky} how problems with {\it more than} $S^1$-symmetries (e.g. possessing a symmetry group of the form $S^1 \times \Gamma$ for some compact Lie group $\Gamma$) are able to manifest bifurcations from relative equilibria into quasiperiodic wave solutions. Under condition \ref{a1}, the complex GLe \eqref{eq:system_GL} admits the symmetry group $\bt^3 := O(2) \times S^1 \times S^1$ with $O(2)$ representing the symmetries of the disc and the two copies of $S^1$ representing temporal and rotational symmetries, respectively. After substituting our relative equilibria ansatz, we are left with the two-torus symmetries of \eqref{eq:system} and solutions of the form
\begin{align} \label{def:spiral_wave}
 u(r,\theta) = e^{im \theta} v(r), \quad m \in \bz, \; v(r) \in \bc,
\end{align}
which are called {\it $|m|$-armed spiral waves} if $m \neq 0$ and {\it target waves}, otherwise. Moreover, the {\it orientation} of a spiral wave solution \eqref{def:spiral_wave}
is said to be {\it clockwise} if $m > 0$ and {\it counterclockwise}, otherwise. 
\begin{remark}
The wave front of a solution \eqref{def:spiral_wave} is determined by its radial profile $v(r)$. While general conditions on the nonlinearity can lead to complex-valued $v(r)$ resulting in visually apparent spiral structures (e.g., resembling the Archimedean spiral shapes illustrated in Figure \ref{fig:1}, wherein solutions of the form $u(r,\theta) = e^{im\theta}e^{ikr}$ are plotted for various $(m,k) \in \bz \times \bz$), the radial part can also be real-valued or have zero phase gradient, corresponding to non-twisting patterns. Our analysis encompasses all possible spiral-like solutions of the form \eqref{def:spiral_wave}.
\end{remark}
\vs
\begin{figure}
    \centering
\includegraphics[width=1\linewidth]{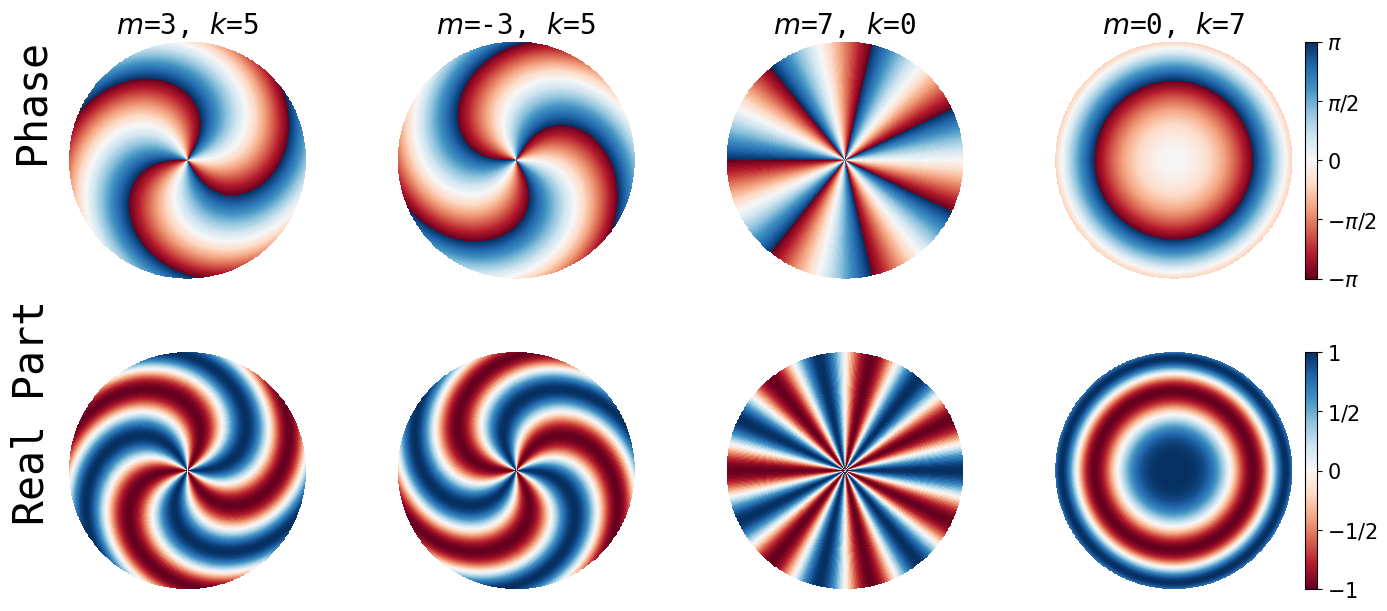}
    \caption{Phase and real parts of the function $u(r,\theta) = e^{im\theta}e^{ikr}$ for various choices of $(m,k) \in \bz \times \bz$.}
    \label{fig:1}
\end{figure} 
For a comprehensive review of the complex GLe
in the first three dimensions with a particular focus on spiral wave behavior from a condensed-matter perspective, we refer the reader to \cite{Aranson}. In one of the many theoretical frameworks synthesized by Aranson et al. in this review, the stability of two-dimensional spiral wave solutions to a GLe defined on a bounded domain is inferred from the stability of the one-dimensional planar waves they emit in an unbounded one-dimensional medium. Complementing this theoretical perspective, Guzmán-Velázquez et al. use finite element simulations in \cite{Guzman} to numerically investigate 
the validity of these analytical predictions for isolated spiral solutions on a bounded circular domain, providing crucial evidence that stability in this setting differs qualitatively from the stability analytically suggested by the associated planar waves. For instance, they show that the 
Eckhaus criterion, a cornerstone for predicting one-dimensional wave persistence, is not a valid predictor for the stability of two-dimensional spiral waves in a finite domain. Taken together, these works reveal an apparent mismatch between the theory for unbounded systems and the observed dynamics in finite domains, highlighting the need for the rigorous analytical results obtained in this paper.
\vs
The existence of branches of spiral wave solutions with fixed temporal frequency $\beta \in \br$, positive orientation and any finite number of arms emerging from the trivial solution to a class of one-parameter, stationary GLes of the form 
\begin{align*} 
\frac{1}{\lambda}(1+i \eta) \Delta u = u - |u|^2u - ib|u|^2u - i \beta u, \quad b \in \br,
\end{align*}
at an infinite sequence of positive critical bifurcation parameter values $0 < \lambda_0 < \lambda_1 < \cdots < \lambda_k < \cdots$ was proved by Dai in \cite{Dai} using a combination of perturbation and shooting arguments. Moreover, in the sequel article \cite{DaiLap}, Dai et al. are able to $(\rm i)$ describe global bifurcation diagrams, $(\rm ii)$ prove the persistence of branches under parameter perturbation and $(\rm iii)$ explicitly construct the global attractor for the branch of spiral wave solutions emerging from any critical point $(\lambda_k,0)$ by assuming sufficiently small values of $\eta, b \in \br$.
\vs
The structural constraints on the nonlinearity and the requirement that $|b|, |\eta| \ll 1$ arise from Dai et al.'s reliance on Rabinowitz-type arguments---based on the classical Leray-Schauder degree---which underpin all of global bifurcation theory. Essentially, the difficulty lies in the fact that the eigenspaces of the linearized GLe  operator relevant to spiral wave solutions are inherently complex. Standard topological degree arguments depend on detecting a change in the sign of an index (e.g. the Leray-Schauder degree) as bifurcation parameters cross critical points, typically guaranteed when eigenvalues cross zero with odd multiplicity. This guarantee invariably fails for eigenvalues with even-dimensional eigenspaces, complicating the global tracking of solution branches.
To circumvent this problem, Dai et al. consider a continuation from the real stationary Ginzburg-Landau equation (corresponding to the case $b = \eta = 0$) where the classical Leray-Schauder degree is a more appropriate tool. 
\vs
In contrast, this paper utilizes the equivariant degree, a topological tool specifically designed for symmetric bifurcation problems.
Our approach is particularly well-suited for studying  \eqref{eq:system} because: $(\rm i)$ the problem possesses inherent $\mathbb{T}^2$ symmetry; $(\rm ii)$ the irreducible representations of the abelian group $\mathbb{T}^2$ are one-dimensional complex spaces, simplifying aspects of the degree computation; and $(\rm iii)$ the values of the $\mathbb{T}^2$-equivariant degree are classified by the isotropy subgroups of $\mathbb{T}^2$ with one-dimensional Weyl groups, which correspond directly to the symmetries of the spiral and target wave solutions sought. In this way, the $\bt^2$-equivariant degree is able to detect and track bifurcating branches of solutions with specific symmetries, even when the Leray-Schauder degree provides insufficient information. Our primary goal is to leverage this tool to establish the existence of unbounded branches of these spiral-like solutions under the general conditions \ref{a1}-\ref{a3}.
\vs
The remainder of this paper is organized as follows: in Section \ref{sec:reformulation} we reformulate our problem in a suitable functional space $\mathscr H$ as the fixed point equation of a nonlinear operator $\mathscr F:\mathscr H \rightarrow \mathscr H$ in the form of a $\bt^2$-equivariant compact perturbation of identity and
in Section \ref{sec:local_global_bif}, we recall equivariant analogues of the classical Krasnosel'skii and Rabinowitz theorems and apply  equivariant degree theory methods to establish local and global bifurcation results for \eqref{eq:system}. 
\vs
The emergence of branches of non-trivial solutions from the trivial solution $(\alpha,\beta,0)$ is only possible at critical parameter values $(\alpha_0,\beta_0) \in \br \times \br$ for which the linearization of $\mathscr F$ around the origin in $\mathscr H$ becomes singular. For equation \eqref{eq:system}, these values are determined by the properties of Bessel functions related to the Neumann boundary conditions on the unit disk. Let $J_m(x)$ be the $m$-th Bessel function of the first kind and denote by $s_{m,n}$ the $n$-th non-negative zero of its derivative $J'_m(x)$. We prove in Section \ref{sec:G_isotypic_decomp} that these critical parameter values can be uniquely associated with an index pair $(m,n) \in \bz \times \bn$ (here $\bn := \{0,1,2,\ldots\})$ via the notation:
    \[
    \alpha_{m,n}:= -s_{|m|,n}, \quad \beta_{m,n} := m \omega - \eta s_{|m|,n}.
    \]
In Section \ref{sec:computation_local_bif_inv}, we formulate our main local bifurcation result, Theorem \ref{thm:main_local_bif}, establishing
that each critical point $(\alpha_{m,n},\beta_{m,n},0) \in \br \times \br \times \mathscr H$ is a branching point for a branch of spiral wave solutions with $|m|$ arms and 
in Section \ref{sec:rab_alt}, we prove our main global bifurcation result, Theorem \ref{thm:main_global_bif} (stated below),  guaranteeing the unboundedness of these branches:
\begin{theorem}\label{thm:main_global_bif}
For each fixed $m \in \bz$, the trivial solution to \eqref{eq:system} at each of the critical parameter values $(\alpha_{m,n},\beta_{m,n}) \in \br \times \br$ with $n \in \bn$
is a branching point for an unbounded branch of non-trivial solutions, consisting of $|m|$-armed spiral waves with orientation determined by the value of $\sign m$ if $m \neq 0$ and target waves, otherwise.
\end{theorem}
\section{Functional Space Reformulation of \eqref{eq:system}} \label{sec:reformulation}
In this section, we prepare the bifurcation problem \eqref{eq:system} for application of the $\bt^2$-equivariant degree with a two-parameter operator equation reformulation in a suitable functional space. Let's begin by considering the Sobolev space 
\begin{align}\label{def:space_H}
\mathscr H := \left\{ u \in H^2(D; \bc): \frac{\partial u}{\partial n}|_{\partial D} = 0 \right\},
\end{align}
equipped with the norm
\begin{align} \label{def:sobolev_norm}
\|u\|_{2}:=\max\{\|D^s u\|_{L^2}: |s|\le 2\},
\end{align}
where $s := (s_1,s_2)$, $|s|:=s_1+s_2 \leq 2$ and $D^s u :=\frac{\partial ^{|s|} u}{\partial r^{s_1} \partial \theta^{s_2}}$. 
Every function $u\in\mathscr H$ admits a complex Fourier expansion of the form
\begin{align}\label{def:fourier_expansion_H}
u(r,\theta) := \sum_{m \in \bz}  R_{m}(r)e^{im\theta}c_m, \quad c_m \in \bc.
\end{align}
The radial component corresponding to the $m$-th Fourier mode $R_m: \br \rightarrow \br$ can be decomposed into an orthogonal series of eigenfunctions $\{ u_{m,n} : \br \rightarrow \br\}_{n \in \bn}$ which are solutions to the eigenproblem
\begin{align}\label{eq:eigenproblem_operatorL}
\begin{cases}
    -\Delta u_{m,n}(r)T_m(\theta) = s_{m,n} u_{m,n}(r)T_m(\theta);   \\
    u_{m,n}'(0) = u_{m,n}'(1) = 0,
\end{cases}
\end{align}
where $T_m(\theta):= e^{im\theta}c_m$ and $\Delta u := (\partial_r^2 + \frac{1}{r}\partial_r + \frac{1}{r^2}\partial_\theta^2) u$. It is well known that \eqref{eq:eigenproblem_operatorL} reduces to the Bessel equation
\begin{align*}
    \begin{cases}
        r^2u^{''}_{m,n}(r) + r u^{'}_{m,n}(r) + \left( s_{m,n}r^2 - m^2 \right)u_{m,n}(r) = 0; \\
        u_{m,n}'(0) = u_{m,n}'(1) = 0,
    \end{cases}
\end{align*}
which has solutions of the form: 
\[
u_{m,n}(r) := J_{|m|}\left( \sqrt{s_{|m|,n}} r \right), \quad (m,n) \in \bz \times \bn,
\]
where $J_{m}: \br \rightarrow \br$ is the $m$-th Bessel function of the first kind. 
From here, the boundary conditions $u_{m,n}'(0) = u_{m,n}'(1) = 0$
imply that eigenvalue $s_{|m|,n}$ associated with the index pair $(m,n) \in \bz \times \bn$ must be the $n$-th non-negative zero of the derivative $J'_{|m|}(x)$. The special case $(m,n)=(0,0)$ corresponds to the zero eigenvalue $s_{0,0} = 0$ with the constant eigenfunction $u_{0,0}(r)\equiv1$. Consequently, the Fourier expansion \eqref{def:fourier_expansion_H} can be further decomposed as follows
\begin{align}\label{def:full_fourier_expansion_H}
u(r,\theta) := \sum_{(m,n) \in \bz  \times \bn} J_{|m|}\left(\sqrt{s_{|m|,n}}r\right) 
e^{im \theta}c_{m,n}, \quad c_{m,n} \in \bc.
\end{align}
Choosing $q > \max\{1,2 c\}$ (cf. Assumption \ref{a3}), we also consider the Nemytski operator
\[
N_f: L^q(D;\bc) \rightarrow L^2(D;\bc), \quad N_f(u)(r, \theta ) := f(u(r, \theta )),
\]
the Banach embeddings
\begin{align*}
j_1: H^1(D;\bc) \hookrightarrow L^2(D;\bc), \quad j_2: \mathscr H \hookrightarrow L^q(D;\bc),
\end{align*}
and the shifted Laplacian operator
\begin{align*}
    \mathscr L : \mathscr H \rightarrow L^2(D; \bc), \quad \mathscr L u := -\Delta u + u.
\end{align*}
Since $N_f$ is continuous, $j_1,j_2$ are compact and $\mathscr L$ is a linear isomorphism, the two-parameter family of operators $\mathscr F: \br \times \br \times \mathscr H \rightarrow \mathscr H$ given by
\[
\mathscr F(\alpha,\beta,u):= u -  \frac{1}{1+i\eta} \mathscr L^{-1}\big( j_1 \circ (1-\alpha + i(\eta- \beta) + \omega \partial_\theta)u -N_f \circ j_2(u) \big),
\]
is a compact perturbation of the identity for every parameter pair $(\alpha,\beta) \in \br \times \br$. Notice also that \eqref{eq:system} is equivalent to the operator equation
\begin{align}\label{eq:operator_equation}
    \mathscr F(\alpha,\beta,u) = 0,
\end{align}
in the sense that a function $u \in \mathscr H$ is a solution to \eqref{eq:system} for a particular parameter pair $(\alpha,\beta) \in \br \times \br$ if and only if $(\alpha,\beta,u) \in \br \times \br \times \mathscr H$ satisfies \eqref{eq:operator_equation}.
\begin{remark} \label{rm:constant_sol_bif} \rm
The bifurcation corresponding to the zero eigenvalue provides a simple but illuminating illustration of our general results. Substituting a constant solution $u_0 \in \bc$ into \eqref{eq:system}, our PDE becomes 
$0 = (\alpha+i\beta)u_0 + f(u_0)$. A bifurcation of constant solutions from the zero solution can only occur when the corresponding linearization, $(\alpha+i\beta)u_0=0$, has non-trivial solutions, corresponding exactly to the critical point $(\alpha_{0,0}, \beta_{0,0}) =(0,0)$.  The bifurcating branch emerging from the origin in $\br \times \br \times \mathscr H$ consists of constant solutions satisfying $f(u_0)=-(\alpha+i\beta)u_0$. By assumption \ref{a1}, if $u_0$ is a solution, so is $e^{i\varphi}u_0$ for any $\varphi \in S^1$. This branch of solutions has the isotropy $S^1 \times \{1\}$, which corresponds to the orbit type $(H_0)$ in our classification (see \eqref{def:subgroup_identification}), perfectly illustrating our theory in its simplest case. For a typical cubic nonlinearity like $f(u)=-|u|^2u$, this yields a circle of solutions with amplitude $|u_0|=\sqrt{\alpha}$ for $\alpha>0$ and $\beta=0$.
\end{remark}
\section{Local and Global Bifurcation in \eqref{eq:system}} \label{sec:local_global_bif}
In this section we assemble the framework of a $\bt^2$-equivariant degree approach for solving two-parameter symmetric bifurcation problems of the form \eqref{eq:operator_equation}. 
\vs
Notice that $\mathscr H$ is a natural Hilbert $\bt^2$-representation with respect to the isometric $\bt^2$-action given by
\begin{align}\label{def:isometric_G_action}
(e^{i \vartheta}, e^{i \varphi}) u(r, \theta ) := e^{i \varphi}u(r, \theta + \vartheta), \quad (e^{i \vartheta}, e^{i \varphi}) \in S^1 \times S^1.
\end{align}
Moreover, under assumptions \ref{a1}--\ref{a2}, the two-parameter family of operators $\mathscr F: \br \times \br \times \mathscr H \rightarrow \mathscr H$ $\rm (i)$ is $\bt^2$-equivariant with respect to the group action \eqref{def:isometric_G_action}, $\rm (ii)$ satisfies $\mathscr F(\alpha,\beta,0) = 0$ 
for all parameter pairs $(\alpha,\beta) \in \br \times \br$ and $\rm (iii)$ is differentiable at $0 \in \mathscr H$ with $\mathscr A(\alpha,\beta) := D \mathscr F(\alpha, \beta,0): \mathscr H \rightarrow \mathscr H$ given by
\begin{align} \label{def:operator_A}
\mathscr A(\alpha,\beta)u = u - \frac{1}{1+i\eta} \mathscr L^{-1}\big( (1-\alpha + i(\eta- \beta) + \omega \partial_\theta)u\big).
    \end{align}
The set of all solutions to the operator equation \eqref{eq:operator_equation} can be divided into the set of {\it trivial solutions:}
\[
M := \{ (\alpha,\beta,0) \in \br \times \br \times \mathscr H \},
\]
and the set of {\it non-trivial solutions:}
\[
\mathscr S := \{ (\alpha,\beta,u) \in \br \times \br \times \mathscr H: \mathscr  F(\alpha,\beta,u) = 0, \; u \neq 0 \}.
\]
Moreover, given any orbit type $(H) \in \Phi_1(\bt^2)$ (cf. Appendix \ref{sec:appendix_eqdeg} for a classification of the isotropy lattice $\Phi_1(\bt^2)$) we can always consider the $H$-fixed-point set 
\[
\mathscr S^H := \{ (\alpha,\beta,u) \in \mathscr S : G_u \geq H \},
\]
consisting of all non-trivial solutions to \eqref{eq:operator_equation} with {\it symmetries at least $(H)$}, i.e. $(\alpha,\beta,u) \in \mathscr S^H$ if and only if $\mathscr F(\alpha,\beta,u) = 0$, $u \in \mathscr H \setminus  \{0\}$ and
\[
h u(r, \theta ) = u(r, \theta ) \text{ for all } h \in H \text{ and } (r, \theta ) \in D.
\]
\subsection{The Local Bifurcation Invariant and Krasnosel'skii's Theorem}\label{sec:local-bif-inv}
For simplicity of notation, we identify $\br \times \br$ with the complex plane $\bc$ by associating each pair of parameters $(\alpha,\beta) \in \br \times \br$ with the complex number $\lambda := \alpha + i \beta$. In order to formulate a Krasnosel'skii-type local equivariant bifurcation result for the equation \eqref{eq:operator_equation}, it will be necessary to introduce a lexicon of bifurcation terminology, following \cite{book-new}.
\vs
First, we clarify what is meant by a bifurcation of the equation \eqref{eq:operator_equation}:
\begin{definition} \rm
    A trivial solution $(\lambda_0,0) \in M$ is said to be a {\it bifurcation point} for \eqref{eq:operator_equation} if every open neighborhood of the point $(\lambda_0,0)$ has a non-empty intersection with the set of non-trivial solutions $\mathscr S$.
\end{definition}
It is well-known that a necessary condition for a trivial solution $(\lambda,0) \in M$ to be a bifurcation point for the equation \eqref{eq:operator_equation} is that the linear operator $\mathscr A(\lambda): \mathscr H \rightarrow \mathscr H$ is {\it not} an isomorphism. This leads to the following definition:
\begin{definition} \rm
    A trivial solution $(\lambda_0,0) \in M$ is said to be a \textit{regular point} for \eqref{eq:operator_equation} if $\mathscr A(\lambda_0): \mathscr H \rightarrow \mathscr H$ is an isomorphism and a \textit{critical point} otherwise. We call the set of all critical points 
\begin{align} \label{def:critical_set}
    \Lambda := \{ (\lambda,0) \in \br^2 \times \mathscr H : \mathscr A(\lambda): \mathscr H \rightarrow \mathscr H \text{ is not an isomorphism}\},   
\end{align}
the {\it critical set} for \eqref{eq:operator_equation}. A critical point $(\lambda_0,0) \in \Lambda$ is said to be \textit{isolated} if there exists an $\varepsilon > 0$ neighborhood $B_{\varepsilon}(\lambda_0) := \{ (\lambda,0) \in \bc \times \mathscr H : \vert \lambda - \lambda_0 \vert < \varepsilon \}$ with
    \[ \overline{B_{\varepsilon}(\lambda_0)}\cap \Lambda = \{ (\lambda_0,0) \}.
    \]
\end{definition}
The next two definitions concern the continuation and symmetric properties of non-trivial solutions emerging from the critical set.
\begin{definition} \rm
A non-empty set $\mathcal C \subset \mathscr S$ is called a {\it branch} of non-trivial solutions to \eqref{eq:operator_equation} if there exists a connected component $\mathcal D$ of $\overline{\mathscr S}$ for which $\mathcal C = \mathscr S \cap \mathcal D$, in which case, any trivial solution $(\alpha_0,0)  \in M$ satisfying $(\alpha_0,0) \in \overline{\mathcal C}$ is said to be a branching point for $\mathcal C$. 
\end{definition}
\begin{definition}  \rm
For a given subgroup $H \leq G$, a non-empty set $\mathcal C \subset \mathscr S^H$ of non-trivial solutions admitting symmetries at least $(H)$ is called {\it a branch of non-trivial solutions to \eqref{eq:operator_equation} with symmetries at least $(H)$} if there exists a connected component $\mathcal D$ of $\overline{\mathscr S^H}$ for which $\mathcal C = \mathscr S^H \cap \mathcal D$.
\end{definition}
Having attended to these necessary preliminaries, let $(\lambda_0,0) \in \Lambda$ be an isolated critical point for \eqref{eq:operator_equation} with a deleted $\varepsilon$-neighborhood 
\begin{align} \label{def:deleted_epsilon_regular_nbhd}
\{ (\lambda,0) \in \br \times \br \times \mathscr H : 0 < \vert \lambda - \lambda_0 \vert < \varepsilon \},    
\end{align}
on which $\mathscr A(\lambda) : \mathscr H \rightarrow \mathscr H$ is an isomorphism, and choose $\delta > 0$ small enough such that
\begin{align} \label{eq:lbi_admissibility_condition}
\{ (\lambda,u) \in \br \times \br \times \mathscr H : | \lambda - \lambda_0 | = \varepsilon, \; \Vert u \Vert_{\mathscr H} \leq \delta \} \cap \overline{\mathscr S} = \emptyset.    
\end{align}
We call the $\bt^2$-invariant set
\begin{align} \label{def:isolating_cylinder}
 \mathscr O := \{ (\lambda,u) \in \br \times \br \times \mathscr H : | \lambda - \lambda_0 | < \varepsilon, \; \Vert u \Vert_{\mathscr H} < \delta \},   
\end{align}
an {\it isolating cylinder} at $(\lambda_0,0)$, and 
a $\bt^2$-invariant function $\Theta: \br \times \br \times \mathscr H \rightarrow \br$ is said to be an {\it auxiliary function} on $\mathscr O$ if it satisfies
\begin{align} \label{def:auxiliary_function_conditions}
   \begin{cases}
\Theta(\lambda,0) < 0 \quad & \text{ for } |\lambda - \lambda_0 | = \varepsilon; \\
\Theta(\lambda,u) > 0  \quad & \text{ for } |\lambda - \lambda_0 | \leq \varepsilon \text{ and } \Vert u \Vert_{\mathscr H} = \delta.
\end{cases}    
\end{align}
\begin{remark} \rm \label{rm:example_auxilliary_function}
For example, we can always use the auxiliary function
\begin{align} \label{def:auxiliary_function}
    \Theta(\lambda,u) := \frac{\varepsilon}{2} - |\lambda - \lambda_0 | + \frac{2\varepsilon}{\delta} \| u \|_{\mathscr H}.
\end{align}   
\end{remark}
Given any auxiliary function $\Theta$, the {\it complemented operator}
\begin{align} \label{def:complemented_operator_F}
\mathscr F_\Theta : \br \times \br \times \mathscr H \rightarrow \br \times \mathscr H, \quad   \mathscr F_\Theta(\lambda,u):= (\Theta(\lambda,u),\mathscr F(\lambda,u)),
\end{align}
is an $\mathscr O$-admissible $\bt^2$-map (cf. \eqref{eq:lbi_admissibility_condition}, Appendix \ref{sec:appendix_eqdeg}). We can now define the {\it local bifurcation invariant at $\lambda_0$}, as follows
\begin{align} \label{def:local_bif_inv}
\omega_{\bt^2}(\lambda_0) := \t2deg(\mathscr F_\Theta,\mathscr O),
\end{align}
where $\t2deg$ indicates the twisted $\bt^2$-equivariant degree (cf. Appendix \ref{sec:appendix_eqdeg}). The reader is referred to \cite{book-new,AED} for proof that the above construction of the local bifurcation invariant is independent of our choice of auxiliary function $\Theta$. On the other hand, the following Krasnosel'skii-type local bifurcation result is a direct consequence of the existence property for the $\bt^2$-equivariant degree (cf. Appendix \ref{sec:appendix_eqdeg} and, in particular, \eqref{def:coefficient_operator_notation} for discussion  of the notation `$\operatorname{coeff}^H$' in the context of the $\bz$-module $A_1(\bt^2):= \bz[\Phi_1(\bt^2)]$).
\begin{theorem} \rm \label{thm:abstract_local_bif} {\bf (Krasnosel'skii's Theorem)}
Suppose that $(\lambda_0,0) \in \Lambda$ is an isolated critical point for \eqref{eq:operator_equation}. If there is an orbit type $(H) \in \Phi_1(\bt^2)$ for which 
\[
\operatorname{coeff}^H(\omega_{\bt^2}(\lambda_0)) \neq 0,
\]
then there exists a branch $\mathcal C$ of non-trivial solutions to \eqref{eq:operator_equation} bifurcating from $(\lambda_0,0)$ with symmetries at least $(H)$.
\end{theorem}
\subsection{The $\bt^2$-Isotypic Decomposition of $\mathscr H$}\label{sec:G_isotypic_decomp}
In order to effectively make use of Theorem \ref{thm:abstract_local_bif} to determine the existence of a branch of non-trivial solutions to \eqref{eq:operator_equation} bifurcating from the zero solution, we must derive a more practical formula for the computation of the local bifurcation invariant \eqref{def:local_bif_inv}. Our first step in this direction is to describe 
the $\bt^2$-isotypic decomposition of $\mathscr H$, i.e. a decomposition of our functional space into the direct sum of irreducible $\bt^2$-representations. 
\vs
As demonstrated in Appendix \eqref{sec:appendix_eqdeg}, every irreducible $\bt^2$-representation can be described in terms of the irreducible $S^1$-representations. In particular, if for each $m \in \bz$, we denote by $\mathcal U_m \simeq \bc$ the irreducible $S^1$-representation equipped with the {\it $m$-folded $S^1$-action}
\begin{align} \label{def:S1_action}
e^{i \vartheta}w := e^{i m\vartheta} \cdot w,  \quad e^{i \vartheta} \in S^1, \; w \in \mathcal U_m,
\end{align}
where `$\cdot$' indicates the standard complex multiplication and by $\mathcal U_0 \simeq \br$ the irreducible $S^1$-representation on which $S^1$ acts trivially, then 
the list of irreducible $\bt^2$-representations consists of the trivial $\bt^2$-representation $\mathcal V_{\bm 0} \simeq \br$ and, for each $\bm k \in \bz_0^2:=\{(k_1,k_2) \in \bz^2 : \text{ if } k_1 = 0 \text{ then } k_2 \geq 0\} \setminus \{\bm 0\}$, the irreducible $\bt^2$-representation $\mathcal V_{\bm k} \simeq \bc$ given by
\[
\mathcal V_{\bm k} := \mathcal U_{k_1} \otimes \mathcal U_{k_2}, \quad \bm k = (k_1,k_2),
\]
and equipped with the corresponding $\bt^2$-action
\[
(e^{i\vartheta},e^{i \varphi}) u := e^{i k_1\vartheta}e^{i k_2\varphi} \cdot u, \quad u \in \mathcal V_{\bm k}.
\]
\begin{remark}
Although the irreducible $S^1$-representations $\mathcal U_m$ and $\mathcal U_{-m}$ are equivalent for all $m =1,2,\ldots$, notice that two irreducible $\bt^2$-representations $\mathcal V_{\bm k}$ and $\mathcal V_{\bm k'}$ with $\bm k, \bm k' \in \bz_0^2$ are equivalent if and only if $\bm k = \bm k'$. 
For more details, we refer the reader to Appendix \ref{sec:appendix_eqdeg}.
\end{remark}
Recalling the Fourier expansion \eqref{def:full_fourier_expansion_H}, let's define the $\bt^2$-invariant subspaces
\begin{align*}
    \mathscr E_{m,n} := \{ J_{|m|}(\sqrt{s_{|m|,n}} r)e^{im \theta}a : a \in \bc \} \quad (m,n) \in  \bz \times \bn;
\end{align*}
and equip each $\mathscr E_{m,n} \subset \mathscr H$ with the corresponding $\bt^2$-action
\begin{align*}
(e^{i \vartheta},e^{i \varphi}) u(r,\theta) =  e^{i \varphi}u(r,\theta + m\vartheta), \quad u \in \mathscr E_{m,n},
\end{align*}
such that one has $\mathscr E_{m,n} \simeq \mathcal U_m \otimes \mathcal U_1$ for all $(m,n) \in \bz \times \bn$.
Consequently, the $\bt^2$-isotypic decomposition of $\mathscr H$ can  now be described in terms of the $\bt^2$-isotypic components $\mathscr H_{m} := \overline{\bigoplus_{n \in \bn } \mathscr E_{m,n}}$ as follows
\begin{align}\label{eq:G_isotypic_decomp}
    \mathscr H = \overline{\bigoplus_{m\in \bz} \mathscr H_{m}}.
\end{align}
To be clear, each $\bt^2$-isotypic component $\mathscr H_{m}$ is modeled on the corresponding irreducible $\bt^2$-representation $\mathcal V_{m,1} := \mathcal U_m \otimes \mathcal U_1$. For easy identification of the relevant symmetry subgroups (cf. Appendix \ref{sec:appendix_eqdeg}), we adopt the notation
\begin{align}\label{def:subgroup_identification}
  H_m := \{ (e^{im\varphi},e^{-i\varphi}) \in S^1 \times S^1 : \varphi \in[0,2\pi] \} \leq \bt^2, \quad m \in \bz,  
\end{align}
such that the full isotropy lattice becomes 
\[
\Phi_1(\bt^2;\mathscr H) = \bigcup_{m \in \bz} \Phi_1(\bt^2;\mathscr H_m), \quad \Phi_1(\bt^2;\mathscr H_m) := \{(H_m) \}.
\]
\begin{remark}\label{rm:orbit_type_identification}
Notice that each $(H_m) \in \Phi_1(\bt^2;\mathscr H)$ is maximal in the sense that if $(H) \in \Phi_1(\bt^2;\mathscr H)$ is such that $(H) \geq (H_m)$, then it must be the case that $H = H_m$. Therefore, any branch $\mathcal C$ of non-trivial solutions with symmetries at least $(H_m) \in \Phi_1(\bt^2;\mathscr H)$ consists only of solutions $u \in \mathscr H$ satisfying $\bt^2_u = H_m$. Notice also that the isotropy subgroup  $\bt^2_u \leq \bt^2$ associated with a non-trivial function $u \in \mathscr H \setminus \{0\}$ satisfies the relation $\bt^2_u \geq H_{m}$ if and only if for all $\varphi \in [0,2\pi]$, one has
\begin{align*}
   e^{im\varphi}u(r,\theta - \varphi) = u(r,\theta).
\end{align*}
In particular, setting $\varphi = \theta$, we find that $u$ must satisfy the relation
\[
u(r,\theta) = e^{im\theta}u(r,0).
\]
In other words, each element of the fixed point space $\mathscr S^{H_m}$ is a spiral wave with $|m|$ arms admitting orientation corresponding to the value of $\sign m$ if $m \neq 0$ and a target wave, otherwise. 
\end{remark}
With the $\bt^2$-isotypic decomposition of $\mathscr H$ at hand, we can begin to collect the spectral data related to the $\bt^2$-equivariant linear operator $\mathscr A(\alpha, \beta): \mathscr H \rightarrow \mathscr H$. For example, we are guaranteed, by Schur's Lemma, that $\mathscr A(\alpha,\beta)$ respects the $\bt^2$-isotypic decomposition \eqref{eq:G_isotypic_decomp} in the sense that
\begin{align*}
    \mathscr A(\alpha,\beta)(\mathscr E_{m,n}) \subset \mathscr E_{m,n}, \quad (m,n) \in \bz \times \bn.
\end{align*}
Therefore, $\mathscr A(\alpha,\beta)$ admits the following block-matrix decomposition
\begin{align*} 
\mathscr A(\alpha,\beta) =  
\bigoplus_{m \in \bz} \bigoplus_{n \in \bn  } \mathscr A_{m,n}(\alpha,\beta), \quad \mathscr A_{m,n}(\alpha,\beta) := \mathscr A(\alpha,\beta)|_{\mathscr E_{m,n}}: \mathscr E_{m,n} \rightarrow \mathscr E_{m,n},
\end{align*}
such that the spectrum of $\mathscr A(\alpha,\beta)$ is given by
\begin{align*}
    \sigma(\mathscr A(\alpha,\beta)) = \bigcup_{m \in \bz} \bigcup_{n\in \bn} \sigma(\mathscr A_{m,n}(\alpha,\beta)).
\end{align*}
In particular, we find by direct computation that the spectrum associated with each matrix
$\mathscr A_{m,n}(\alpha,\beta)$ is comprised of the complex eigenvalue
\begin{align} 
\label{def:eigenvalues_Anm}
\mu_{m,n}(\alpha,\beta) := \frac{(1+i\eta)s_{|m|,n} + \alpha +i \beta - i \omega m}{(1+i\eta)(1+s_{|m|,n})}.
\end{align}
\begin{remark} \label{rm:critical_values} \rm
Recall that a trivial solution $(\alpha_0,\beta_0,0) \in \br \times \br \times \mathscr H$ belongs to the critical set of \eqref{eq:operator_equation} if and only if $0 \in \sigma(\mathscr A(\alpha_0,\beta_0))$. In other words, $(\alpha_0,\beta_0,0) \in \Lambda$ if and only if 
there exist $(m,n) \in \bz \times \bn$
for which 
\begin{align*}
\begin{cases}
\alpha_0 = -s_{|m|,n}; \\
\beta_0 = \omega m  -\eta s_{|m|,n}.
\end{cases}
\end{align*}
Notice that each eigenvalue $\mu_{m,n}: \bc \rightarrow \bc$ admits exactly one root since the zeros of $J'_{|m|}: \br \rightarrow \br$ form a strictly increasing sequence $s_{|m|,0} < s_{|m|,1} < \cdots < s_{|m|,n} < \cdots$. On the other hand, two eigenvalues $\mu_{m,n}$ and $\mu_{m',n'}$ share the same root if and only if
\begin{align*}
 \begin{cases}
s_{|m|,n} = s_{|m'|,n'}; \\
\omega m  - \eta s_{|m|,n} = \omega m'  - \eta s_{|m'|,n'},
\end{cases}   
\end{align*}
which holds if and only if $m = m'$ and $s_{|m|,n} = s_{|m|,n'}$. Again, by strict monotonicity of the sequence  $\{s_{|m|,n}\}_{n \in \mathbb N }$, notice that the latter condition implies $n = n'$. Therefore, each critical point can be uniquely associated with an index pair $(m,n) \in \bz \times \bn$ using the notation 
\begin{align}\label{def:critical_point_identification}
\lambda_{m,n} := \alpha_{m,n} + i \beta_{m,n}, \quad  (\alpha_{m,n}, \beta_{m,n}) :=  \left(-s_{|m|,n}, \; \omega m  - \eta s_{|m|,n} \right),
\end{align}
such that the critical set, in addition to being discrete, admits the following explicit description
\[
\Lambda = \{ (\alpha_{m,n},\beta_{m,n},0) : m \in \bz, \;  n \in  \bn  \}.
\]
\end{remark}
\subsection{Computation of the Local Bifurcation Invariant} \label{sec:computation_local_bif_inv}
As before, let $(\lambda_0,0) \in \Lambda$ be a critical point for \eqref{eq:operator_equation} with a deleted $\varepsilon$-neighborhood \eqref{def:deleted_epsilon_regular_nbhd} on which $\mathscr A(\lambda):\mathscr H \rightarrow \mathscr H$ is an isomorphism and suppose that a number $\delta > 0$ is chosen such that
\[
\mathscr F(\lambda,u) \neq 0, \text{ for all } (\lambda,u) \in \br \times \br \times \mathscr H \text{ with } |\lambda - \lambda_0| = \epsilon \text{ and } 0 < \| u \|_{\mathscr H} \leq \delta.
\]
Then, for any auxiliary function $\Theta:\br \times \br \times \mathscr H \rightarrow \br$ satisfying conditions \eqref{def:auxiliary_function_conditions} on the isolating cylinder \eqref{def:isolating_cylinder} (in particular, for the auxiliary function \eqref{def:auxiliary_function}), the complemented operator \eqref{def:complemented_operator_F} is $\mathscr O$-admissibly $\bt^2$-homotopic to the linear operator
\begin{align*}
    \hat{\mathscr F_\Theta}(\lambda,u) := (\Theta(\lambda,u),\mathscr A(\lambda)u)  = \left(\Theta(\lambda,u), \bigoplus_{m\in \bz}\bigoplus_{n\in\bn} \mathscr A_{m,n}(\lambda)  \right).
\end{align*}   
Adopting the notations
\begin{align*}
\begin{cases}
\wt{\Theta}_{m,n}(\lambda,u) :=        \Theta|_{\bc \times \mathscr E_{m,n}}(\lambda,u);\\ \wt{\mathscr A}_{m,n}(\lambda)u := \left(\wt{\Theta}_{m,n}(\lambda,u), \mathscr A_{m,n}(\lambda)u\right); \\ \wt{\mathscr O}_{m,n} :=  \mathscr O \cap \left(\bc \times \mathscr E_{m,n}\right),
\end{cases}
\end{align*}
and combining the homotopy property of the $\bt^2$-equivariant degree (see the third degree axiom in Appendix \ref{sec:appendix_eqdeg}) with the Splitting Lemma (see Lemma \eqref{lemm:splitting_lemma}, Appendix \ref{sec:appendix_comp_form}), the local bifurcation invariant \eqref{def:local_bif_inv} at the isolated critical point $(\lambda_0,0)$ becomes
\begin{align}
\label{eq:local_bif_inv_reformulation}
\omega_{\bt^2}(\lambda_0) &= \t2deg\left( 
\bigoplus_{m \in \bz} \bigoplus_{n \in \bn}  \wt{\mathscr A}_{m,n}, \bigoplus_{m \in \bz} \bigoplus_{n \in \bn} \wt{\mathscr O}_{m,n} \right) \\ &= 
\sum_{m \in \bz} \sum_{n \in \bn} \t2deg(\wt{\mathscr A}_{m,n},  \wt{\mathscr O}_{m,n}). \nonumber
\end{align}
Since $\id - \mathscr A(\alpha,\beta): \mathscr H \rightarrow \mathscr H$ is compact for all $(\alpha,\beta) \in \br \times \br$, the eigenvalues $\mu_{m,n}(\alpha_0,\beta_0)$ are nonzero for almost all index pairs $(m,n) \in \bz \times \bn$ such that, with only finitely many exceptions, one has
\begin{align*}
 \t2deg( \wt{\mathscr A}_{m,n}, \wt{\mathscr O}_{m,n}) = 0.
\end{align*}
Indeed, by Lemma \ref{lemm:analytic_formula}, the $\bt^2$-equivariant degree of each complemented operator $(\theta_{m,n}, \mathscr A_{m,n})$ on its corresponding isolating neighborhood 
$\wt{\mathscr O}_{m,n}$ is fully specified by the spectrum of $\mathscr A_{m,n}(\alpha,\beta)$ and the irreducible $\bt^2$-representation $\mathcal V_{m,1}$ (resp. $\mathcal V_0$, in the case that $m=0$) according to the rule:
\begin{align*}
        \t2deg( \wt{\mathscr A}_{m,n}, \wt{\mathscr O}_{m,n}) = 
\begin{cases}
    \rho_{m,n}(\alpha_0,\beta_0) (H_{m}) \quad & \text{ if } \mu_{m,n}(\alpha_0,\beta_0) = 0; \\
        0 & \text{ otherwise,}
    \end{cases}
\end{align*}   
where $\rho_{m,n}(\alpha,\beta) := \deg( \det\nolimits_\bc \mathscr A_{m,n}, B_\varepsilon(\lambda))$. 
\begin{lemma} \label{lemm:local_bif_inv_comp_formula} \rm
Using the notation \eqref{def:critical_point_identification}, the local bifurcation invariant at any critical point $(\lambda_{m,n},0) = (\alpha_{m,n},\beta_{m,n}) \in \Lambda$ is given by the rule
\begin{align}
    \omega_{\bt^2}(\lambda_{m,n}) = (H_m).
\end{align}
\end{lemma}
\begin{proof}
Since each critical point is isotopically simple, i.e. since one has $\mu_{m,n}^{-1}(0) = \{ (\alpha_{m,n},\beta_{m,n})\}$ for all $(m,n) \in \bz \times \bn$ (cf. Remark \ref{rm:critical_values}), the computational formula \eqref{eq:local_bif_inv_reformulation} simplifies to
\begin{align*}
    \omega_{\bt^2}(\lambda_{m,n}) = \t2deg( \wt{\mathscr A}_{m,n}, \wt{\mathscr O}_{m,n}).
\end{align*}
At this point the result follows from the observation that, for all $(\alpha,\beta) \in \br \times \br$, the Jacobian matrix
\begin{align} \label{eq:Jacobian_mu_mnj}
 D \mu_{m,n}(\alpha,\beta) = 
 \frac{1}{(1+i\eta)(s_{|m|,n}+1)}\begin{pmatrix*}[r]
1 & 0 \\  0 & 1
\end{pmatrix*}, \end{align}
is non-singular such that the local Brouwer degrees $\rho_{m,n}(\alpha,\beta)$ are always well-defined and can be computed as follows
\begin{align} \label{eq:rho_mn_formula}
    \rho_{m,n}(\alpha,\beta)
    &= \deg( \det\nolimits_\bc \mathscr A_{m,n},B_{\varepsilon}(\lambda) ) \\ &= \deg(\mu_{m,n}, B_{\varepsilon}(\lambda)) = 
\operatorname{sign}\left[ \det  D \mu_{m,n}(\alpha,\beta)\right]  \nonumber \\ 
 &= \operatorname{sign}\left[  \det
 \begin{pmatrix*}[r]
1 & 0 \\ \nonumber  0 & 1
\end{pmatrix*}\right] = 1. \nonumber
\end{align}
\end{proof}
\vs
We are now in a position to formulate our main local equivariant bifurcation result.
\begin{theorem} \rm \label{thm:main_local_bif}
For each fixed $m \in \bz$, the trivial solution to \eqref{eq:system} at each of the critical parameter values $(\alpha_{m,n},\beta_{m,n}) \in \br \times \br$ with $n \in \bn$ is a branching point for a branch of non-trivial solutions, consisting of $|m|$-armed spiral waves with orientation determined by the value of $\sign m$ if $m \neq 0$ and target waves, otherwise. 
\end{theorem}
\begin{proof}
A direct consequence of the combination of Theorem \ref{thm:abstract_local_bif}, Remark \ref{rm:orbit_type_identification}, Lemma \ref{lemm:local_bif_inv_comp_formula} and the existence property of the $\bt^2$-equivariant degree (cf. Appendix \ref{sec:appendix_eqdeg}).
\end{proof}
\subsection{Resolution of The Rabinowitz Alternative} \label{sec:rab_alt}
We can study the global behaviour of the branches of spiral wave solutions whose existence has been predicted by Theorem \ref{thm:main_local_bif} using the following Rabinowitz-type argument, the proof of which can be found in \cite{book-new, AED}:
\begin{theorem} \rm \label{thm:rab_alt} {\bf (Rabinowitz' Alternative)}
Let $\mathcal U \subset \br \times \br \times \mathscr H$ be any open bounded $\bt^2$-invariant set with $\partial \mathcal U \cap \Lambda = \emptyset$. If $\mathcal C \subset \mathscr S$ is a branch of non-trivial solutions to \eqref{eq:operator_equation} bifurcating from a critical point $(\lambda_0,0) \in \mathcal U \cap \Lambda$, then one has the following alternative: 
\begin{enumerate}[label=$(\alph*)$]
\item\label{alt_a}  either $\mathcal C \cap \partial \mathcal U \neq \emptyset$;
    \item\label{alt_b} or there exists a finite set
    \begin{align*}
        \overline{\mathcal C} \cap \Lambda = \{ (\lambda_0,0),(\lambda_1,0), \ldots, (\lambda_{k},0) \},
    \end{align*}
    satisfying the following relation
    \begin{align*}
    \sum\limits_{i=0}^{k} \omega_{\bt^2}(\lambda_i) = 0.
    \end{align*}
\end{enumerate} 
\end{theorem}
\begin{remark} \rm
    If a branch of non-trivial solutions $\mathcal C \subset \overline{\mathscr S}$ satisfies $\mathcal C \cap \partial \mathcal U \neq \emptyset$ for every open bounded $\bt^2$-invariant set $\mathcal U \subset \br \times \br \times \mathscr H$ with $\partial \mathcal U \cap \Lambda = \emptyset$, then  $\mathcal C$ must be {\it unbounded}. Therefore, a sufficient condition for the unboundedness of a branch $\mathcal C \subset \overline{\mathscr S}$ is the following:
        \begin{align*}
        \sum\limits_{(\lambda_i,0) \in\overline{\mathcal C} \cap \Lambda} \omega_{\bt^2}(\lambda_i) \neq 0.
    \end{align*}
\end{remark}
We now have all the necessary components to prove our main global equivariant bifurcation result:
\begin{proof}[Proof of Theorem~\ref{thm:main_global_bif}]
Notice that the coefficient standing next to $(H_m)$ in the local bifurcation invariant at any particular $(\lambda_{m',n'},0) \in \Lambda$ is determined by the rule
\begin{align*}
\operatorname{coeff}^{H_m}\left(\omega_{\bt^2}(\lambda_{m',n'})\right)=    \begin{cases}
        1 & \text{ if } m' = m; \\
        0 & \text{ otherwise}.
    \end{cases}
\end{align*}
We notice again (cf. Theorem \ref{thm:main_local_bif}) that emerging from the trivial solution at each critical point is a branch of non-trivial solutions $\mathcal C$ with corresponding spiral or target pattern symmetries. To show that these branches are unbounded, suppose for contradiction that a branch $\mathcal{C}$ emerging from the critical point $(\lambda_{m,n},0) \in \Lambda$ is bounded. Then $\overline{\mathcal{C}} \cap \Lambda$ is a finite set of critical points, including $(\lambda_{m,n},0)$. By the Rabinowitz alternative, this would require
 \begin{align*} \operatorname{coeff}^{H_m}\left( \sum_{(\lambda_{m',n'},0)\in \overline{\mathcal C} \cap \Lambda} \omega_{\bt^2}(\lambda_{m',n'}) \right) = | (\lambda_{m',n'},0) \in \overline{\mathcal C} \cap \Lambda : m' = m \}| =0, \end{align*}
    which is in contradiction with the assumption that $(\lambda_{m,n},0) \in \overline{\mathcal C} \cap \Lambda$.
\end{proof}

\newpage
\appendix
\section{The $\bt^n$-Equivariant Degree}\label{sec:appendix_eqdeg}
\noi{\bf Equivariant notation:}
In what follows, we indicate by $\bt^n$ the $n$-dimensional torus and, for any subgroup  $H \leq \bt^n$, we denote by $(H)$ its conjugacy class in $\bt^n$.
Notice that the set of all subgroup conjugacy classes $\Phi(\bt^n):= \{(H) : H \leq \bt^n \}$ admits the natural partial ordering:
\[
(H)\leq (K) \iff\;\; H\leq K.
\]
As is possible with any partially ordered set, we extend the natural order over $\Phi(\bt^n)$ to a total order, which we indicate by $\preccurlyeq$ to differentiate the two relations. We also put $\Phi_1(\bt^n):= \{ (H) \in \Phi(\bt^n) :  \dim \bt^n/H = 1\}$ and $A_1(\bt^n) := \mathbb{Z}[\Phi_1(\bt^n)]$, noticing that every element of the of the free $\mathbb{Z}$-module $A_1(\bt^n)$ can be expressed as a formal sum over some finite number of generator elements  
\[
a = n_1(H_1) + n_2(H_2) + \cdots + n_m(H_m), \quad a \in A_1(\bt^n).
\]
In particular, we can specify the integer coefficient standing next to the generator element $(H) \in \Phi_1(\bt^n)$ in any element $a \in A_1(\bt^n)$ using the notation
\begin{align}\label{def:coefficient_operator_notation}
\operatorname{coeff}^H(a) = n_H.
\end{align}
Given a $\bt^n$-space $X$ with an element $x \in X$, we denote by
$\bt^n_{x} :=\{g\in \bt^n:gx=x\}$ the {\it isotropy group} of $x$
and we call $(\bt^n_{x}) \in \Phi(\bt^n)$ the {\it orbit type} of $x \in X$. Moreover, we put $\Phi(\bt^n,X) := \{(H) \in \Phi(\bt^n)  : 
(H) = (\bt^n_x) \; \text{for some $x \in X$}\}$ and also  $\Phi_1(\bt^n,X):= \Phi(\bt^n,X) \cap \Phi_1(\bt^n)$. For a subgroup $H\leq \bt^n$, the subspace $
X^{H} :=\{x\in X:\bt^n_{x}\geq H\}$ is called the {\it $H$-fixed-point subspace} of $X$. If $Y$ is another $\bt^n$-space, then a continuous map $f : X \to Y$ is said to be {\it $\bt^n$-equivariant} if $f(gx) = gf(x)$ for each $x \in X$ and $g \in \bt^n$.
\noi{\bf Classification of the Real Irreducible $\bt^n$-Representations:} Since every complex irreducible representation of an abelian group is one-dimensional, each irreducible $\bt^n$-representation $\mathcal V$ can be identified with a continuous homomorphism $T:\bt^n \rightarrow S^1$. Notice also that the set of homomorphisms from the torus $\bt^n$ to the circle group $S^1$  is related to the group $\bz^n$ via the isomorphism
\begin{align}\label{def:isomorphism}
   (z, \bm k) \rightarrow z_1^{k_1}z_2^{k_2}\cdots z_n^{k_n}, \quad z = (z_1,z_2,\ldots,z_n) \in \bt^n, \; \bm k := (k_1,k_2,\ldots,k_n)  \in \bz^n. 
\end{align}
In particular, we can identify every element in $\operatorname{Hom}(\bt^n,S^1)$ with a corresponding vector $\bm k \in \bz^n$ using the notation $T_{\bm k}:\bt^n \rightarrow S^1$ to indicate the map \eqref{def:isomorphism}.
\vs
While any two complex irreducible $\bt^n$-representations $T_{\bm k}$ and $T_{\bm k'}$ are equivalent if and only if $\bm k = \bm k'$, the irreducible $\bt^n$-representations $T_{\bm k}$ and $T_{-\bm k}$ are always equivalent as real representations. It follows that the set of all non-trivial, real irreducible $\bt^n$-representations is in one-to-one correspondence with the set 
\[
\bz_0^n := \{(k_1,k_2,\ldots,k_n) \in \bz^n \setminus \{ \bm 0\}: \text{ if } k_1 = k_2 = \ldots = k_i = 0 \text{ then } k_{i+1} \geq 0 \}.
\]
Having removed possible duplicate indices, we can identify the list of all irreducible real $\bt^n$-representations, using the notation $\mathcal V_{\bm 0} \simeq \br$ to indicate the trivial $\bt^n$-representation and $\mathcal V_{\bm k}$ to indicate the irreducible $\bt^n$-representation corresponding to the homomorphism $T_{\bm k}$ with $\bm k \in \bz_0^n$ (see \eqref{def:isomorphism}), i.e. the irreducible $\bt^n$-representations
\[
\mathcal V_{\bm k} \simeq \mathcal U_{k_1} \otimes \mathcal U_{k_2} \otimes \cdots \otimes \mathcal U_{k_n}, \quad \bm k = (k_1,k_2,\ldots,k_n) \in \bz_0^n, 
\]
where, for each $m > 0$ we denote by $\mathcal U_m \simeq \bc$ the irreducible $S^1$-representation equipped with the {\it $m$-folded} $S^1$-action $e^{i \vartheta} w := e^{i m \vartheta} \cdot w$, 
$\mathcal U_{-m} \simeq \overline{\mathcal U_m}$ its conjugated counterpart with the 
{\it reverse $m$-folded} $S^1$-action $e^{i \vartheta} w := e^{-i m \vartheta} \cdot w$ and $\mathcal U_0 \simeq \br$ the irreducible $S^1$-representation on which $S^1$ acts trivially.
\vs
\noi{\bf Classification of the set $\Phi_1(\bt^n)$:} 
Given a real non-trivial irreducible $\bt^n$-representation $\mathcal V_{\bm k}$, consider the normal subgroup $\Gamma_{\bm k} \leq \bt^n$ given by $\Gamma_{\bm k} := \ker{T_{\bm k}}$, i.e. the subgroup
\[
\Gamma_{\bm k} = \{(z_1,z_2,\ldots,z_n) \in \bt^n: z_1^{k_1}z_2^{k_2}\cdots z_n^{k_n} = 1\}.
\]
Since $T_{\bm k}: \bt^n \rightarrow S^1$ is surjective, one always has $\bt^n/\Gamma_{\bm k} \simeq S^1$ such that $(\Gamma_{\bm k}) \in \Phi_1(\bt^n)$. Moreover, for any non-zero element $x_0 \in \mathcal V_{\bm k} \setminus \{0\}$, notice from
\begin{align*}
  \bt^n_{x_0} = \{ (z_1,z_2,\ldots,z_n) \in \bt^n: z_1^{k_1}z_2^{k_2}\cdots z_n^{k_n} x_0 = x_0 \},
\end{align*}
that the associated isotropy subgroup $\bt^n_{x_0} \leq \bt^n$ coincides with the group $\Gamma_{\bm k}$. Therefore, $\Phi_1(\bt^n, \mathcal V_{\bm k})$ is always the singleton set $\{(\Gamma_{\bm k})\}$ and the isotropy lattice $\Phi_1(\bt^n,V)$ in the more general case that
$V$ is any orthogonal $\bt^n$-representation with the isotypic decomposition
\[
V = \mathcal V_{\bm k_1} \oplus \mathcal V_{\bm k_2} \oplus \cdots \oplus \mathcal V_{\bm k_m},
\]
is such that
\[
\Phi_1(\bt^n,V) \supseteq \{ (\Gamma_{\bm k_1}), (\Gamma_{\bm k_2}), \ldots, (\Gamma_{\bm k_m}) \}.
\]
\vs
\noi{\bf The $\bt^n$-Equivariant Degree.}
Let $V$ be an orthogonal $\bt^n$-representation and consider the induced representation $\br \times V$ where $\bt^n$ acts trivially on $\br$. An open bounded $\bt^n$-invariant set $\Om \subset \br \times V$ together with a $\bt^n$-equivariant map
$f: \br \times V \rightarrow V$ constitute an {\it admissible $\bt^n$-pair in $\br \times V$} if  $f(x) \neq 0$ for all $x \in \partial \Om$, in which case the map $f$ is said to be {\it $\Om$-admissible}. We denote by $\mathcal M_1^{\bt^n}(V)$ the set of all admissible $\bt^n$-pairs in $\br \times V$ and by $\mathcal M_1^{\bt^n}$ the set of all admissible $\bt^n$-pairs defined by taking a union over all orthogonal $\bt^n$-representations, i.e.
\[
\mathcal M_1^{\bt^n} := \bigcup\limits_V \mathcal M_1^{\bt^n}(V).
\]
The $\bt^n$-equivariant degree provides an algebraic count of solution orbits, according to their symmetric properties, to equations of the form
\[
f(x) = 0, \; x \in \Omega,
\]
where $(f, \Omega) \in \mathcal M_1^{\bt^n}$. In fact, using arguments analogous to those used in the definition of the $S^1$-equivariant degree (cf. \cite{AED}, \cite{book-new}) we define the {\it$\bt^n$-equivariant degree} as the unique map associating to every admissible $\bt^n$-pair $(f,\Om)\in \mathcal M_1^{\bt^n}$ an element from the free $\bz$-module $A_1(\bt^n)$, satisfying the four {\it degree axioms} of existence, additivity, homotopy and normalization:
\vs
\begin{theorem} 
\label{thm:GpropDeg} There exists a unique map $\tndeg:\mathcal{M}_1
	^{\bt^n}\to A_1(\bt^n)$, that assigns to every admissible $\bt^n$-pair $(f,\Omega)$ the module element
	\begin{equation}
		\label{eq:G-deg0}\tndeg(f,\Omega)=\sum_{(H) \in \Phi_1(\bt^n)}%
		{n_{H}(H)},
	\end{equation}
	satisfying the following properties:
	\begin{itemize}
		\item[] \textbf{(Existence)} If  $n_{H} \neq0$ for some $(H) \in \Phi_1(\bt^n)$ in \eqref{eq:G-deg0}, then there
		exists $x\in\Omega$ such that $f(x)=0$ and $(\bt^n_{x})\geq(H)$.
		\item[] \textbf{(Additivity)} 
  For any two  disjoint open $\bt^n$-invariant subsets
  $\Omega_{1}$ and $\Omega_{2}$ with
		$f^{-1}(0)\cap\Omega\subset\Omega_{1}\cup\Omega_{2}$, one has
		\begin{align*}	\tndeg(f,\Omega)=\tndeg(f,\Omega_{1})+\tndeg		(f,\Omega_{2}).
	\end{align*}	
		\item[] \textbf{(Homotopy)} For any 
  $\Omega$-admissible $\bt^n$-homotopy, $h:[0,1]\times \br \times V\to V$, one has
		\begin{align*}	\tndeg(h_{t},\Omega)=\mathrm{constant}.
		\end{align*}
\item[] \textbf{(Normalization)} Let $(f, \Omega) \in \mathcal{M}_1^{\bt^n}$ be such that $f$ is regular normal in $\Omega$ (meaning $(\rm i)$ $f$ is smooth, $(\rm ii)$ $f^{-1}(0) \cap \Omega$ consists of orbits $(\bt^n(w_i))$ such that $f$ is transversal to $\{0\}$ along each orbit, and $(\rm iii)$ for every $H \leq \bt^n$ such that $(H) = (\bt^n_{w_i})$ for some $i$, $0$ is a regular value of $f^H|_{\Omega^H}$ - cf. \cite{book-new} for a formal definition of regular normality). Assume further that $f^{-1}(0) \cap \Omega = \bt^n(w_0)$ consists of a single orbit for some $w_0 \in \Omega$. Then,
\begin{align*}
    \tndeg(f, \Omega) =
    \begin{cases}
    \rho_0 (\bt^n_{w_0}) & \text{if } (\bt^n_{w_0}) \in \Phi_1(\bt^n); \\
    0 & \text{otherwise},
    \end{cases}    
\end{align*}
where 
\begin{align}\label{def:orbit_index}
  \rho_0 := \operatorname{sign} \det(Df(w_0)|_{S_{w_0}}),
\end{align}
and $S_{w_0}$ is the positively oriented slice---that is, the orthogonal subspace in $\br \times V$--- to the orbit $\bt^n(w_0)$ at $w_0$.
\end{itemize}
\end{theorem}
\subsection{Two Computational Formulae for two-parameter $\bt^n$-Equivariant Degree Calculations} \label{sec:appendix_comp_form}
Given a non-trivial irreducible $\bt^n$-representation $\mathcal V_{\bm k} \simeq \bc$, we can always put $V_{\bm k}:= \br \times \mathcal V_{\bm k}$ and identify $\br \times V_{\bm k}$ with the space $\bc \times \mathcal V_{\bm k}$. Let $\mu_{\bm k}: S^1 \rightarrow GL^{G}(\mathcal V_{\bm k})$ be a continuous family of $G$-equivariant invertible linear operators and define the family of $\bt^n$-equivariant complemented linear operators
\[
\mathcal A_{\bm k}: \br^2 \times \mathcal V_{\bm k} \rightarrow \br \times \mathcal V_{\bm k}, \quad \mathcal A_{\bm k}(\lambda)v := \left(1 - |\lambda|, \mu_{\bm k}\left(\frac{\lambda}{|\lambda|}\right) \cdot v\right).
\]
Since the action of $\bt^n$ on $\mathcal V_{\bm k}$ is isometric, the open bounded set
\[
\mathcal D_{\bm k} := \{ (\lambda,v) \in \br^2 \times \mathcal V_{\bm k} : \| v \| < 1, \; \frac{1}{2} < | \lambda| < 2 \},
\]
is $\bt^n$-invariant. Moreover, since the system
\begin{align*}
    \begin{cases}
        \mu_{\bm k}\left(\frac{\lambda}{|\lambda|}\right) \cdot v = 0\\
        1 - | \lambda| = 0,
    \end{cases}
\end{align*}
admits no solutions on $\partial \mathcal D_{\bm k}$, the pair $(\mathcal A_{\bm k}, \mathcal D_{\bm k})$ constitutes an admissible $\bt^n$-pair in $\br \times  V_{\bm k} \simeq \bc \times \mathcal V_{\bm k}$. In this way, we are able to employ the $\bt^n$-equivariant degree for solving two-parameter bifurcation problems.
\vs
The proofs of the following pair analytical formulae are based on identical arguments used to prove analogous results for the $S^1$ degree in \cite{book-new}, \cite{AED}, and for this reason they are omitted.
\begin{lemma} \label{lemm:analytic_formula}
For any admissible pair $( \mathcal A_{\bm k}, \mathcal D_{\bm k}) \in \mathcal M^{\bt^n}_1(\mathcal V_{\bm k})$ constructed in the above manner, one has
    \[
    \tndeg( \mathcal A_{\bm k}, \mathcal D_{\bm k}) = \deg(\det\nolimits_\bc(\mu_{\bm k})) (\Gamma_{\bm k}).
    \]
\end{lemma}
\begin{lemma}\label{lemm:splitting_lemma} {\bf (The Splitting Lemma):}
For any two admissible $G$-pairs $(\mathcal A_{\bm k},\mathcal D_{\bm k})$ and $(\mathcal A_{\bm k'},\mathcal D_{\bm k'})$ constructed in the above manner, one has
\[
\tndeg(\mathcal A_{\bm k} \times \mathcal A_{\bm k'},D_{\bm k} \times \mathcal D_{\bm k'}) =  \tndeg(\mathcal A_{\bm k}, D_{\bm k}) + \tndeg(\mathcal A_{\bm k'}, D_{\bm k'}).
\]
\end{lemma}

\end{document}